\documentclass[11pt]{article}
    
\usepackage{amsmath,amsfonts,amsthm,amscd,amssymb,graphicx}
\usepackage{xcolor}

\usepackage{subfigure}

\numberwithin{equation}{section}

\usepackage{hyperref}

\newtheorem{theorem}{Theorem}[section]

\newtheorem{lemma}[theorem]{Lemma}

\newcommand{\RR}{\mathbb{R}}

\newcommand{\ZZ}{{\mathbb Z}}
\newcommand{\TT}{{\mathbb T}}

\def\beq{\begin{equation}}
\def\eeq{\end{equation}}
\def\bb1{{1\!\!1}}
\def\rit{{\Bbb R}}

\begin{document}

\title{Generator functions and their applications}

\author{Emmanuel Grenier\footnotemark[1]
  \and Toan T. Nguyen\footnotemark[2]
}

\maketitle

\renewcommand{\thefootnote}{\fnsymbol{footnote}}

\footnotetext[1]{Equipe Projet Inria NUMED,
 INRIA Rh\^one Alpes, Unit\'e de Math\'ematiques Pures et Appliqu\'ees., 
 UMR 5669, CNRS et \'Ecole Normale Sup\'erieure de Lyon,
               46, all\'ee d'Italie, 69364 Lyon Cedex 07, France. Email: Emmanuel.Grenier@ens-lyon.fr}

\footnotetext[2]{Department of Mathematics, Penn State University, State College, PA 16803. Email: nguyen@math.psu.edu. TN's research was partly supported by the NSF under grant DMS-1764119, an AMS Centennial fellowship, and a Simons fellowship.}


\subsubsection*{Abstract}

In \cite{GrN6}, we introduced so called {\em generators} functions to precisely follow the regularity of analytic solutions
of Navier Stokes equations. 
In this short note, we give a presentation of these generator functions and use them to give existence results of
analytic solutions to some classical equations, namely to hyperbolic  equations, to incompressible Euler equations, 
 and to hydrostatic Euler and Vlasov models.
The use of these generator functions appear to be an alternative way to the use of the classical
abstract Cauchy-Kovalevskaya theorem \cite{Asano,Caflisch,Nirenberg,Safonov}.


\section{Introduction}


The general abstract Cauchy-Kovalevskaya theorem has been intensively used to construct analytic solutions 
to various evolution partial differential equations, including first order hyperbolic and parabolic equations,
 or Euler and Navier-Stokes equations. We refer for instance to Asano \cite{Asano}, Caflisch \cite{Caflisch}, 
 Nirenberg \cite{Nirenberg}, Safonov \cite{Safonov}, among others.  
 In this short note, we use generators functions as introduced in \cite{GrN6} to obtain existence results for these
 equations. The results are not new, but show the versatility and simplicity of use of these generator functions.
 We believe that such approaches could be used on many other equations and provide easy ways to
 get existence results or to investigate instabilities.
 
Let us first introduce generator functions in the particular case of a periodic function $f(t,x)$ 
on $t \ge 0$ and $x\in \TT^d$, $d\ge 1$. For $z \in \rit$, we introduce the {\em generator function} $Gen[f]$ defined by
\beq \label{Genbl}
\begin{aligned}
Gen[f](t,z) &= \sum_{\alpha \in \ZZ^d} e^{z |\alpha|} |f_\alpha(t) | 
\end{aligned}\eeq
in which $f_\alpha$ denotes the Fourier transform of $f(t,x)$ with respect to $x\in \TT^d$. 
If $f$ is analytic in $x$, $Gen[f]$ is only defined for small enough $|z|$, up to the analyticity radius of $f(t,.)$. 
The results in this note also apply to the case when $x\in \RR^d$, 
with which the above summation is replaced by the integral over $\RR^d$. 
In applications, we may also introduce generator functions depending on multi-variables $z = (z_1,\cdots ,z_d)$ 
that correspond to the analyticity radius of $f(x)$ in $(x_1,\cdots, x_d)$, respectively; 
see, for instance, \cite{GrN6} for the case of boundary layers on the half space $\TT \times \RR_+$.

First note that generator functions are non negative, and that all their derivatives are non negative and non decreasing in $z$.
Moreover generator functions have very nice properties with respect to algebraic operations and differentiation.
Namely they "commute" with the product, the sum and the differentiation, making their use very versatile. 

\begin{lemma}\label{lem-generators} For any $f,g$, there hold the following properties
\begin{equation}\label{Gen-property}
Gen [f + g] \le  Gen[f] + Gen[g]
\end{equation}
\beq \label{Gen-property2}
Gen [fg] \le  Gen[f] Gen[g] 
\eeq
\beq \label{Gen-property3}
Gen[\nabla_xf] = \partial_zGen[f] ~~\qquad 
\end{equation}
\beq \label{Gen-property4}
\partial_t Gen[f] \le Gen[\partial_t f] ~~\qquad 
\end{equation}
for all $z\ge 0$. 
\end{lemma}

\begin{proof}  Let $f_\alpha$ and $g_\alpha$ be the Fourier transform of $f$ and $g$, respectively.
It follows that 
$$ 
(fg)_\alpha = f_\alpha \star_\alpha g_\alpha
$$ 
for $\alpha \in \ZZ^d$. 
For $z\ge 0$, we compute 
$$\begin{aligned}
Gen[fg](z) 
&= \sum_{\alpha \in \ZZ^d} e^{z |\alpha|} |f_\alpha \star_\alpha g_\alpha | 
\le \sum_{\alpha \in \ZZ^d} \sum_{\beta \in \ZZ^d} e^{z |\beta|} e^{z|\alpha-\beta|} |f_{\beta}|  |g_{\alpha-\beta} | 
\\&\le Gen[f] Gen[g]
\end{aligned}$$
which is the second inequality. The first and third identities follow directly from the definition. 
The last inequality is a direct consequence of
$$
\partial_t | f_\alpha | \le | \partial_t f_\alpha |.
$$
\end{proof}

The use of generator functions are not limited to polynomial operations. Namely, we have 
\begin{lemma}\label{lem-analyticF}
 Let $F$ be an analytic function
$$
F(z) = \sum_{n \ge 0} a_n z^n
$$
with some convergence radius $\rho$. Let us introduce
$$
\widetilde F(z) = \sum_{n \ge 0} | a_n | z^n.
$$
Then, for any function $f$, provided $| f | < \rho$, there holds 
\beq
Gen[F(f)] \le \widetilde F (Gen[f]) .
\eeq
\end{lemma}

\begin{proof}
First, using \eqref{Gen-property2}, we have 
$$
Gen[a_n f^n] \le |a_n| Gen[f]^n .
$$
The lemma follows from multiplying the inequality by $z^n$ and summing over $n$. 
\end{proof}

Moreover generator functions easily combine with truncation. Let $P_N$ be the projection of the first $N$ Fourier modes; namely
\begin{equation}\label{def-PN}
P_N(f)(x) = \sum_{\alpha \in \ZZ^d, |\alpha| \le N} e^{i \alpha x} f_\alpha.\end{equation}
Then it follows that 
\beq \label{est-PN}
Gen[P_N(f)] \le Gen[f] .
\eeq
The main point is that if $f(t,x)$ satisfies a non linear partial differential equation then $Gen[f]$ satisfies
a similar differential inequality. This inequality describes in an acute way how the radius of analyticity shrinks
as time goes on, and allows to get analytic bounds on $f$, and in particular to bound all its derivatives at the same
time. We now illustrate this remark in the particular case of hyperbolic equations.


\section{Hyperbolic equations}


We consider the following system of evolution equations
\begin{equation}\label{1st-pde}
\partial_t u = A(u, \nabla_x u)
\end{equation}
for a vector valued function $u = u(x,t)$, with $x\in \TT^d$, $d\ge 1$, and $t\ge 0$. 
We assume that  the function $A(\cdot,\cdot)$ satisfies 
\begin{equation}\label{A-condition}
Gen[A(u,\nabla_x u)] \le C_0 F( Gen[u] ) \Big( 1 + \partial_z Gen[u]\Big)
\end{equation} 
for some constant $C_0$ and some analytic functions $F$.
For instance, $A$ may be a quadratic polynomial in $u$ and $\nabla u$, in which case $F = Id$,
or more generally, $A$ may be of the form $G(u) \cdot \nabla_x u$, for which $F = \widetilde G$ (see Lemma \ref{lem-analyticF}).
 
We shall construct solutions in the function space $X_\rho$ defined by 
\begin{equation}\label{def-Xrho} 
X_\rho : = \Big \{ u~:~ Gen[u](\rho) <\infty \Big\}.
\end{equation}
\begin{theorem}\label{theo-1stpde} 
Let $\rho>0$ and $u_0$ be in $X_{\rho}$. Then  the Cauchy problem \eqref{1st-pde} 
with initial data $u_0$ has a unique solution $u(t) $ in $X_{\rho(t)}$ for positive times $t$ as long as 
$\rho(t) = \rho - C_1 t$ remains positive, $C_1$ being some large positive constant depending on $u_0$. 
\end{theorem}

\begin{proof} We shall construct solutions via the standard approximation. 
First, recall the projection $P_N$ defined as in \eqref{def-PN}. Let $u_N(t)$ be a solution to the following regularized equations 
\begin{equation}\label{1st-pde}
\partial_t u_N = P_N A(P_N(u_N), \nabla_x P_N (u_N))
\end{equation}
with  initial data $u_N(0,x) = P_N u_0(x)$ for all $N$. As the right hand side consists of only a finite number of Fourier modes,
 this equation is in fact an ordinary differential equation. Hence there exists a unique solution $u_N(t)$, defined for $t$ small enough.
 It suffices to prove that $u_N(t)$ is a Cauchy sequence in $X_{\rho(t)}$, as long as $\rho(t)$ remains positive.   	
		
The Fourier coefficients of the solution to \eqref{1st-pde} satisfy
$$ 
\partial_t u_{N,\alpha} =P_{N,\alpha} A(P_N u_N, \nabla_x P_N u_N)
$$
where $P_{N,\alpha} f$ denotes the Fourier coefficient $f_\alpha$ if $|\alpha|\le N$, and zero, if otherwise. 
We therefore get
$$
\begin{aligned}
\sum_{| \alpha | \le N} \partial_t | u_{N,\alpha} | e^{z |\alpha|}
&\le \sum_{| \alpha | \le N} | \partial_t u_{N,\alpha} | e^{z |\alpha|}
\\&
\le  \sum_{| \alpha | \le N}  | P_{N,\alpha} A(P_N u_N, \nabla_x P_N u_N) | e^{z |\alpha|}.
\end{aligned}$$
Using \eqref{est-PN} and the assumption \eqref{A-condition}, we obtain the following Hopf-type differential inequality
\begin{equation}\label{Hopf}
\partial_t Gen[u_N]\le C_0 F(Gen[u_N])  \Big( 1 +  \partial_z Gen[u_N]\Big),
\end{equation}
in which $C_0$ is independent of $N$. For convenience, we set 
$$
G_N(t,z) = Gen[u_N(t)](z).
$$
The previous inequality yields 
\beq \label{diffineq}
\partial_t G_N \le C_0 F(G_N) (1+ \partial_z G_N),
\eeq
which is a differential inequality that we will now exploit in order to bound $u_N$ and all its derivatives.
Note that $G_N$ is a finite sum and is therefore defined for all $z$. As $N$ goes to $+ \infty$, $G_N(0,z)$
converges to $Gen[u_0](z)$, which is defined for $|z| \le \rho$.

As usual with analytic solutions, the domain of analyticity shrinks with time, hence we introduce 
$$
F_N(t,z) = G_N(t,\theta(t) z)
$$
for $t, z\ge 0$, where $\theta(\cdot)$ will be determined later. It follows that $F_N$ satisfies
\begin{equation}\label{ode-FN}
\partial_t F_N \le C_0 F(F_N) + C_0(F( F_N)  + \theta'(t) z) \partial_z F_N.
\end{equation}
Note that $F_N$ is defined for any $z$. In the limit $N \to + \infty$ we focus on $0 \le z \le \rho$.

We will choose $\theta(t)$ is such a way that the characteristics of (\ref{ode-FN}) are outgoing on $[0,\rho]$,
namely such that at $z =0$, $F(F_N) > 0$ (which is always satisfied) and such that at $z =\rho$,
  $F(F_N) + \theta'(t)  \rho < 0$. 

At $t= 0$, we choose $\theta(0) =1$ and thus 
$$
F_N(0,z)  = G_N(0,z) = Gen[P_N u_0](z)\le Gen[u_0](z),
$$
which is well-defined on $[0,\rho]$. 
Set
$$
M_0 = \sup_{0 \le z \le \rho} Gen[u_0](z).
$$
We will focus on times $0 \le t \le T_N$ such that $F_N(t,z) \le 2 M_0$ for $0 \le t \le T_N$ and $0 \le z \le \rho$.
We choose
$$
\theta(t) = 1 - F(3 M_0)  \rho^{-1} t .
$$
Observe that on $0 \le t \le T_N$, as long as $\theta(t) > 0$,  $F(F_N) + \theta'(t) \rho <0$. 
On such a time interval, (\ref{ode-FN})
is a nonlinear transport equation with a source term and with outgoing characteristics at $0$ and $\rho$.
As a consequence, we have 
$$
\partial_t \sup_{0 \le z \le 1} F_N(t,z)  \le C_0 \sup_{0 \le z \le 1} F(F_N(t,z)) \le C_0 F(2 M_0),
$$ 
hence
$$
F_N(t,z) \le M_0 + C_0 t F(2 M_0).
$$
Classical arguments then leads to the fact that $T_N$ is bounded away from $0$, namely there exists some
$T > 0$ such that $T_N \ge T$ for any $N$, and such that $\rho(t) > 0$ for any $t \le T$. 
This implies that 
$$ 
F_N(t,z) \le C(F_N(0,z)) \le C(M_0)
$$
for any $0 \le z \le \rho$ and any $0 \le t \le T$.
Thus $u_N(t)$ is uniformly bounded in $X_{\rho(t)}$ for all $N$.  
As a consequence, $u_N$ and all its derivatives of all orders are uniformly bounded in $L^\infty$. Up to the extraction
of a subsequence, $u_N$ and all its derivatives converge uniformly, towards some function $u$, solution of (\ref{1st-pde}).
Moreover, classical arguments show that $u(t) \in X_{\rho(t)}$ for any $0 \le t \le T$, which ends the proof of the theorem. 
\end{proof}


\section{Euler equations}


In this section, we apply the previous framework to construct analytic solutions
 to incompressible Euler equations on $\TT^d$ or $\RR^d$, $d\ge 2$. Namely, we consider 
\begin{equation}\label{NS1}
\begin{aligned}
 \partial_t u + u \cdot \nabla u + \nabla p &=0
 \\\nabla \cdot u &=0
\end{aligned} \end{equation} 
on $\TT^d$ (the case $\RR^d$ is treated similarly). 
Again, the existence result is classical (see, for instance, \cite{Bardos-B, Bardos-BZ}).
 The function space $X_\rho$ is defined in \eqref{def-Xrho}. We have

\begin{theorem}\label{theo-NS} 
Let $\rho>0$ and $u_0$ be a divergence-free vector field in $X_{\rho}$.
 Then the Cauchy problem \eqref{NS1} with initial data $u_0$ has a solution $u(t) $ in 
 $X_{\rho(t)}$ for positive times $t$ as long as $\rho(t) = \rho - C_1 t $ remains positive, $C_1$ being a constant depending
 on $u_0$. 
\end{theorem}

\begin{proof} 
Introduce the Leray projector $\mathbb{P} $,  projection onto the divergence-free $L^2$ vector space. 
In Fourier coefficients, $\mathbb{P}_\alpha$ is an $d\times d$ matrix with entries 
$$
(\mathbb{P}_\alpha)_{jk} = \delta_{jk} - \frac{\alpha_j\alpha_k}{|\alpha|^2}.
$$ 
In particular, $|\mathbb{P}_\alpha|$ is bounded, uniformly in $\alpha \in \ZZ^d$. 
Taking the Leray projection of \eqref{NS1}, we obtain 
\begin{equation}\label{NS-P} 
\partial_t u  = - \mathbb{P} (u\cdot \nabla u),
\end{equation}
which falls into the previous abstract framework. It remains to check the assumption \eqref{A-condition}. 
Indeed, using \eqref{Gen-property} and the uniform boundedness of $|\mathbb{P}_\alpha|$ in $\alpha$, we compute 
$$ 
Gen[\mathbb{P} (u\cdot \nabla u)] \le Gen[u\cdot \nabla u] \le Gen[u] Gen[\nabla u] \le Gen[u]\partial_z Gen[u].
$$
Theorem \ref{theo-NS} follows from Theorem \ref{theo-1stpde}, 
upon noting that the divergence-free condition is invariant under \eqref{NS-P}.  
\end{proof}


\section{Hydrostatic Euler equation and Vlasov models}


In this section, we apply the abstract framework to two more examples that arise as a singular limit of 
Euler equations and Vlasov-Poisson systems. The first is the hydrostatic Euler equations which read 
\begin{equation}\label{HE1}
\partial_t  \omega +  u\cdot \nabla \omega  = 0,
\end{equation}
on $\TT \times [-1,1]$ for vorticity function 
\begin{equation} \label{HE1bis}
 \omega = \partial_y^2 \varphi 
 \end{equation}
  and velocity field $u = \nabla^\perp \varphi$, with $\varphi$ being the stream function, satisfying the boundary condition 
\begin{equation}\label{HE2}
\varphi_{\vert_{y=\pm 1}} = 0.
\end{equation}
The model \eqref{HE1}-\eqref{HE2} can be derived from the two-dimensional incompressible Euler equations 
in a narrow channel \cite{Lions,Gr99-2,Br03}. 
The Cauchy problem for data with Sobolev regularity is illposed \cite{HKN2}, except for convex profiles \cite{Br99}, 
due to the loss of derivatives in $x$ through $u_2 = -\partial_x \varphi = -\partial_x \partial_y^{-2}\omega$. 
However, the local solution can be constructed for data with analytic regularity \cite{KTVZ}. 

We now show that the previous abstract framework also applies to the hydrostatic Euler model \eqref{HE1}-\eqref{HE2}. 
For $z\ge 0$, we introduce the notion of generator functions as follows: 
\begin{equation}\label{vort-norm0}
Gen[\omega](z): =\sum_{\alpha\in \mathbb{Z}} \sum_{\beta\ge 0} e^{z|\alpha| }  
\|\partial_y^\beta \omega_\alpha\|_{L^\infty[-1,1]} \frac{z^{\beta}}{ \beta!} ,\end{equation}
in which $\omega_\alpha(y)$ denotes the Fourier transform of $\omega(x,y)$ with respect to $x$.
 The analyticity radius for $x$ and $y$ needs not to be the same \cite{GrN6}, 
 however for simplicity, we have taken the generator of the form \eqref{vort-norm0}. We obtain the following. 

 \begin{theorem}\label{theo-HE} 
 Let $\rho>0$ and $\omega_0$ be in $X_{\rho}$. 
 Then, the Cauchy problem \eqref{HE1}-\eqref{HE2} with initial data $\omega_0$ 
 has a solution $\omega(t) $ in $X_{\rho(t)}$ for positive times $t$ as long as 
 $\rho(t) = \rho - 3 C_1 t$ remains positive, $C_1$ being a constant depending on $\omega_0$. 
\end{theorem}
\begin{proof} 
Let us first check the properties of the new generator function. 
For any $f,g$, we compute 
$$
\partial_y^\beta (f g)_\alpha =
\sum_{\alpha'\in \ZZ} \sum_{0 \le \beta' \le \beta} 
{\beta ! \over \beta' ! (\beta - \beta') !}  \partial_y^{\beta'} f_{\alpha'} 
\partial_y^{\beta - \beta'} g_{\alpha-\alpha'},
$$
for any $\alpha \in \ZZ$ and $\beta\ge 0$. Therefore, for $z\ge 0$, we have 
$$
\begin{aligned}
&Gen(fg)(z) 
\\&= \sum_{\alpha \in \ZZ}\sum_{\beta \ge 0} e^{z|\alpha|}\| \partial_y^\beta (fg)_\alpha \|_{L^\infty} {z^\beta \over \beta ! } 
\\
&\le \sum_{\alpha \in \ZZ}\sum_{\beta \ge 0} \sum_{\alpha'\in \ZZ}\sum_{0 \le \beta' \le \beta} 
e^{z|\alpha|} \|\partial_y^{\beta'} f_{\alpha'}\|_{L^\infty}
\| \partial_y^{\beta - \beta'} g_{\alpha-\alpha'} \|_{L^\infty}  
{z^\beta \over \beta' ! (\beta - \beta') !}  
\\
&\le \sum_{\alpha,\alpha'\in \ZZ}\sum_{ \beta \ge 0} \sum_{\beta \ge \beta' } e^{z|\alpha'|} 
e^{z|\alpha-\alpha'|} \|\partial_y^{\beta'} f_{\alpha'}\|_{L^\infty}
\| \partial_y^{\beta - \beta'} g_{\alpha - \alpha'} \|_{L^\infty} 
{z^{\beta'}z^{\beta - \beta'} \over \beta' ! (\beta - \beta') !}  
\\
&\le Gen[f](z) Gen[g](z),
\end{aligned}$$
which proves the second identity in Lemma \ref{lem-generators}. The other identities are straightforward. 

Let us now verify assumption \eqref{A-condition}. In this case 
$$
A(\omega,\nabla \omega) = - u \cdot \nabla \omega,
$$ 
where $u = \nabla^\perp \varphi$, with $\varphi$ defined through \eqref{HE1bis}-\eqref{HE2}. 
We first note that
$$ 
Gen[u \cdot \nabla \omega] \le Gen[u]\partial_z Gen[\omega].
$$
Next, $\partial_y^2 \varphi = \omega$ (with the Dirichlet boundary conditions on $[-1,1]$) can be
explicitly solved, leading to
$$ 
\|\varphi\|_{L^\infty} + \| \partial_y \varphi \|_{L^\infty} + \| \partial_y^2 \varphi\|_{L^\infty} \le C_0 \|\omega\|_{L^\infty}.
$$ 
This yields to
$$
 Gen[u] \le C_0 \Big[ Gen[\omega] + \partial_z Gen[\omega] \Big],
 $$
recalling the loss of derivatives in variable $x$ through $u_2 = -\partial_x \varphi$. As a result, we have obtained 
\begin{equation}\label{est-HE1}
Gen[A(\omega, \nabla \omega)] \le C_0 \Big[ Gen[\omega] + \partial_z Gen[\omega] \Big] \partial_z Gen[\omega] .
\end{equation}
To fit into the previous framework, we introduce the vector functions 
$$ 
U := [\omega, \nabla \omega] , \qquad \mathcal{A}(U,\nabla U) := [A(\omega, \nabla \omega), \nabla A (\omega, \nabla \omega)].
$$
It follows that $\partial_t U = \mathcal{A}(U,\nabla U)$.  
The theorem follows from the previous abstract framework.  
\end{proof}

We conclude our paper with an application to kinetic models, known as the {kinetic incompressible Euler} 
and Vlasov-Dirac-Benney systems. These models are derived from 
the quasineutral limit or the vanishing Debye length limit of Vlasov-Poisson systems 
\cite{Br89, Gr99, HK11, HKFr, Bardos, HKN2}. The Vlasov equation reads
  \begin{equation}
\label{VP1}
\partial_t f + v\cdot \nabla_x f  - \nabla_x \varphi \cdot \nabla_v f = 0, 
\end{equation}
on $\TT^d\times \RR^d$, $d\ge 1$, where the potential function $\varphi$ is defined by
  \begin{equation}
\label{VP2}
 \varphi =  \int_{\mathbb{R}^d} f(t,x,v) \, dv  -1
\end{equation}
for the Vlasov-Dirac-Benney system \cite{Bardos,HK11}, or alternatively, defined through the elliptic equation 
 \begin{equation}
\label{VP3}
-\Delta \varphi =  \nabla \cdot \left(\nabla \cdot \int f v \otimes v \, dv  \right)
\end{equation}
for the kinetic incompressible Euler model \cite{Br89}. Again, the system \eqref{VP1}-\eqref{VP3} experiences 
a loss of derivatives in $x$. As a consequence, 
the Cauchy problem for data with Sobolev regularity is illposed \cite{HKN2}, except for stable data \cite{BB,HKFr}. 
The local analytic solutions were constructed in \cite{HK11,JN,BFJJ}. 
 
We now show that the abstract framework also applies to the kinetic model \eqref{VP1}-\eqref{VP3}. 
Indeed, we introduce the following modified notion of generator functions: 
\begin{equation}\label{vort-norm1}
Gen[f](z): =\sum_{\alpha\in \mathbb{Z}^d} \sum_{\beta\ge 0} e^{z|\alpha| }  \| 
\langle v\rangle^m \partial_v^\beta f_\alpha\|_{L^\infty(\RR^d)} \frac{z^{\beta }}{ \beta!} ,
\end{equation}
for some $m>d+2$, in which $f_\alpha(v)$ denotes the Fourier transform of $f(x,v)$ with respect to $x$, and
$\langle v \rangle = (1 + \| v \|^2)^{1/2}$. 
With exception of the weight in $v$, the generator function is identical to \eqref{vort-norm0}. 
Thus, Lemma \ref{lem-generators} follows similarly. As a result, we have the following 

 \begin{theorem}\label{theo-kinetic} 
Let $\rho>0$ and $f_0$ be in $X_{\rho}$.
Then, the Cauchy problem \eqref{VP1}-\eqref{VP3} with initial data $f_0$ has a solution $f(t) $ in $X_{\rho(t)}$ 
for positive times $t$ as long as $\rho(t) = \rho - 3 C_1 t$ remains positive, where $C_1$ is a constant depending on $f_0$. 
\end{theorem}
\begin{proof}
It remains to verify assumption \eqref{A-condition} for 
$$
A(f,\nabla_v f) = \nabla_x \varphi[f] \cdot \nabla_v f
$$ 
for $\varphi[f]$ defined by the relation \eqref{VP2} or \eqref{VP3}. In the first case, for $\alpha \not =0$, we have 
$$ 
|\varphi_\alpha | \le \int_{\mathbb{R}^d} |f_\alpha(t,v)| \, dv \le C_0 \|\langle v\rangle^m f_\alpha \|_{L^\infty}
$$
for some $m>d$. The second case is similar, with $m>d+2$. This proves that 
$$ 
Gen[\nabla_x \varphi] \le C_0 \partial_z Gen[f] 
$$ 
and hence 
$$ 
Gen[A(f,\nabla_v f)] \le C_0  \partial_z Gen[f]  \partial_z Gen[f] .
$$
Again, as in the previous case, to fit into the abstract framework, we introduce the vector functions 
$$ 
U := [f, \nabla_x f, \nabla_v f] 
$$
$$
\mathcal{A}(U,\nabla U) := [A(f, \nabla_x f), \nabla_xA(f, \nabla_x f), \nabla_v A(f, \nabla_x f)],
$$
which yields $\partial_t U = \mathcal{A}(U,\nabla U)$. 
The theorem follows from Theorem \eqref{theo-1stpde}.
\end{proof}



\bibliographystyle{abbrv}

\begin{thebibliography}{10}

\bibitem{Asano}
K.~Asano.
\newblock A note on the abstract {C}auchy-{K}owalewski theorem.
\newblock {\em Proc. Japan Acad. Ser. A Math. Sci.}, 64(4):102--105, 1988.

\bibitem{Bardos}
C.~Bardos.
\newblock {About a Variant of the $1d$ Vlasov equation, dubbed
  ``Vlasov-Dirac-Benney'' Equation}.
\newblock {\em S{\'e}minaire Laurent Schwartz - EDP et applications}, 15:21 p.,
  2012-2013.

\bibitem{BB}
C.~Bardos and N.~Besse.
\newblock The {C}auchy problem for the {V}lasov-{D}irac-{B}enney equation and
  related issues in fluid mechanics and semi-classical limits.
\newblock {\em Kinet. Relat. Models}, 6(4):893--917, 2013.


\bibitem{Bardos-B}
C. Bardos and S. Benachour.
\newblock Domaine d'analycit\'e des solutions de l'\'equation d'Euler dans un ouvert de $\RR^n$. (French)
\newblock {\em Ann. Scuola Norm. Sup. Pisa Cl. Sci.} (4) 4 (1977), no. 4, 647-687. 

\bibitem{Bardos-BZ}
C. Bardos, S. Benachour, and M. Zerner. 
\newblock Analyticit\'e des solutions p\'eriodiques de l'\'equation d'Euler en deux dimensions. (English summary)
\newblock {\em C. R. Acad. Sci. Paris S\'er.} A-B 282 (1976), no. 17, Aiii, A995-A998. 

\bibitem{BFJJ}
M.~Bossy, J.~Fontbona, P.-E. Jabin, and J.-F. Jabir.
\newblock Local existence of analytical solutions to an incompressible
  {L}agrangian stochastic model in a periodic domain.
\newblock {\em Comm. Partial Differential Equations}, 38(7):1141--1182, 2013.

\bibitem{Br89}
Y.~{Brenier}.
\newblock {A Vlasov-Poisson type formulation of the Euler equations for perfect
  incompressible fluids}.
\newblock {\em {Rapport de recherche INRIA}}, 1989.

\bibitem{Br99}
Y.~Brenier.
\newblock Homogeneous hydrostatic flows with convex velocity profiles.
\newblock {\em Nonlinearity}, 12(3):495--512, 1999.

\bibitem{Br03}
Y.~Brenier.
\newblock Remarks on the derivation of the hydrostatic {E}uler equations.
\newblock {\em Bull. Sci. Math.}, 127(7):585--595, 2003.

\bibitem{Caflisch}
R.~E. Caflisch.
\newblock A simplified version of the abstract {C}auchy-{K}owalewski theorem
  with weak singularities.
\newblock {\em Bull. Amer. Math. Soc. (N.S.)}, 23(2):495--500, 1990.

\bibitem{Gr99}
E.~Grenier.
\newblock Limite quasineutre en dimension 1.
\newblock In {\em Journ\'ees ``\'{E}quations aux {D}\'eriv\'ees {P}artielles''
  ({S}aint-{J}ean-de-{M}onts, 1999)}, pages Exp.\ No.\ II, 8. Univ. Nantes,
  Nantes, 1999.

\bibitem{Gr99-2}
E.~Grenier.
\newblock On the derivation of homogeneous hydrostatic equations.
\newblock {\em M2AN Math. Model. Numer. Anal.}, 33(5):965--970, 1999.

\bibitem{GrN6}
E.~Grenier and T.~T. Nguyen.
\newblock ${L}^\infty$ instability of {P}randtl's layers.
\newblock {\em Annals of PDE}, 2019. 

\bibitem{HK11}
D.~Han-Kwan.
\newblock Quasineutral limit of the {V}lasov-{P}oisson system with massless
  electrons.
\newblock {\em Comm. Partial Differential Equations}, 36(8):1385--1425, 2011.

\bibitem{HKN2}
D.~Han-Kwan and T.~T. Nguyen.
\newblock Ill-posedness of the hydrostatic {E}uler and singular {V}lasov
  equations.
\newblock {\em Arch. Ration. Mech. Anal.}, 221(3):1317--1344, 2016.

\bibitem{HKFr}
D.~Han-Kwan and F.~Rousset.
\newblock Quasineutral limit for {V}lasov-{P}oisson with {P}enrose stable data.
\newblock {\em Ann. Sci. \'Ec. Norm. Sup\'er. (4)}, 49(6):1445--1495, 2016.

\bibitem{JN}
P.-E. Jabin and A.~Nouri.
\newblock Analytic solutions to a strongly nonlinear {V}lasov equation.
\newblock {\em C. R. Math. Acad. Sci. Paris}, 349(9-10):541--546, 2011.

\bibitem{KTVZ}
I.~Kukavica, R.~Temam, V.~C. Vicol, and M.~Ziane.
\newblock Local existence and uniqueness for the hydrostatic {E}uler equations
  on a bounded domain.
\newblock {\em J. Differential Equations}, 250(3):1719--1746, 2011.

\bibitem{Lions}
P.-L. Lions.
\newblock {\em Mathematical topics in fluid mechanics. {V}ol. 1}, volume~3 of
  {\em Oxford Lecture Series in Mathematics and its Applications}.
\newblock The Clarendon Press, Oxford University Press, New York, 1996.
\newblock Incompressible models, Oxford Science Publications.

\bibitem{Nirenberg}
L.~Nirenberg.
\newblock An abstract form of the nonlinear {C}auchy-{K}owalewski theorem.
\newblock {\em J. Differential Geometry}, 6:561--576, 1972.
\newblock Collection of articles dedicated to S. S. Chern and D. C. Spencer on
  their sixtieth birthdays.

\bibitem{Safonov}
M.~V. Safonov.
\newblock The abstract {C}auchy-{K}ovalevskaya theorem in a weighted {B}anach
  space.
\newblock {\em Comm. Pure Appl. Math.}, 48(6):629--637, 1995.

\end{thebibliography}

\end{document}